\documentclass[11pt, oneside]{article}   	
\usepackage{geometry}                		
\usepackage{graphicx}				

\usepackage{indentfirst}

\usepackage{ulem}	
\usepackage{amsthm}
\usepackage{amsfonts}
\usepackage{amsmath}
\usepackage{amscd}
\usepackage{graphicx}
\usepackage{enumerate}
\usepackage{amssymb}	
\usepackage{verbatim}
\usepackage{cancel}
\usepackage{cite}
\usepackage{textcomp}								
\usepackage{color}

\newtheorem{theorem}{Theorem}[section]
\newtheorem{definition}{Definition}[section]
\newtheorem{proposition}{Proposition}[section]
\newtheorem{lemma}{Lemma}[section]
\newtheorem{remark}{Remark}[section]
\newtheorem{question}{Question}[section]
\newtheorem{corollary}{Corollary}[section]
\newtheorem{conjecture}{Conjecture}[section]

\title{\textsc{On density of infinite subsets I}}
\author{\textsc{Changguang Dong}}
\date{}

\begin{document}
\maketitle
\begin{abstract}
Let $Y$ be a compact metric space, $G$ be a group acting by transformations on $Y$. For any infinite subset $A\subset Y$, we study the density of $gA$ for $g\in G$ and quantitative density of the set $\displaystyle{\bigcup_{g\in G_n}gA}$ by the Hausdorff semimetric $d^H$. It is proven that for any integer $n\ge 2$, $\epsilon>0$, any infinite subset $A\subset \mathbb T^n$, there is a $g\in SL(n,\mathbb Z)$ such that $gA$ is $\epsilon$-dense. We also show that, for any infinite subset $A\subset [0,1]$, for generic rotation and generic 3-IET, $$\liminf_nn\cdot d^H\left(\bigcup_{k=0}^{n-1}T^kA,[0,1]\right)=0.$$
\end{abstract}

\section{Introduction and Results}

Let $Y$ be a compact metric space, $G$ be a locally compact second countable (lcsc) (semi-)group. Let $\alpha$ be a $G$ action on $Y$ by transformations. If $\alpha$ admits an ergodic probability measure with full support, then by Poincar\'e recurrence Theorem, it is easy to see that the orbit of almost every point in $Y$ is dense. For some special group actions, it is already an interesting question to consider the effectivization of the density of an orbit. For example in \cite{eff} and \cite{wang}, they investigate how fast the orbit of a generic point can become dense in the torus for certain higher rank abelian actions. In order to obtain the effective result, one has to make some restriction on the single orbit (or finite many orbits). Without making any generic assumptions, it is natural to consider at first a subset of infinite many points, and study the density of the iterations under the group action. Along this way, we are going to describe and study two types of density problems. Throughout this note, $d^H$ is the Hausdorff semimetric, and $d_L$ is the standard metric.

\subsection{Dense iterations of infinite subset}

Let $A$ be an infinite subset of $Y$, we can consider the set containing all subsets of the form $gA:=\{\alpha(g)x|x\in A\}$ for a $g\in G$. For the fixed $A$, we would like to know: for any $\epsilon>0$, whether there exists a $g\in G$ such that $gA$ is $\epsilon$-dense in $Y$, or equivalently $d^H_Y(gA,Y)<\epsilon$. We will call this dense iteration problem simply {\bf D.I. problem}. Before we give some definitions, let's first state a nontrivial result in this direction.

Let $S^1=\mathbb R/\mathbb Z$ be the standard circle, and $T_\alpha:S^1\to S^1$ be the translation map: $x\mapsto x+\alpha\;(mod\, 1)$. A theorem of Glasner \cite{gl} asserts that if $X$ is an infinite subset of $S^1$, then for any $\epsilon>0$, there exists an integer $n$ such that the dilation $nX := \{nx\; (mod\, 1) : x\in X\}$ is $\epsilon$-dense. This gives an affirmative answer to the D.I. problem in the case of the natural action by multiplication of $\mathbb N$ on the circle $S^1$. 

In view of this result, we make the following definitions.

\begin{definition}
Given a $G$ action on a metric space $Y$, if an infinite subset $A$ satisfies that for any $\epsilon>0$, there exists a $g\in G$ such that $gA$ is $\epsilon$-dense in $Y$, then $A$ is called {\bf Glasner set} with respect to $(Y,d,G)$.
\end{definition}

Let's remark that in \cite{bp} and \cite{nv}, the authors have already used the term {\it Glasner set}, but that is defined on the acting group. Our definition here is obviously different from theirs.

\begin{definition}
Given a $G$ action on a metric space $Y$, if any infinite subset $A$ is a Glasner set, then we say the dynamical system $(Y,d,G)$ has {\bf Glasner property}.
\end{definition}

Using our definition, the natural action $(S^1,d_L,\mathbb N)$ has the Glasner property. In \cite{kl}, Kelly and L\^e generalize the result to higher dimensional torus. In particular, they proved that for any positive integer $N$, the system $(\mathbb T^N,d_L,M(N\times N,\mathbb Z))$ has Glasner property.

Our first result is a partial strengthening of the above mentioned one. 

\begin{theorem}\label{tl}
For any integer $n\ge 2$, the system $(\mathbb T^n,d_L,SL(n,\mathbb Z))$ has Glasner property.
\end{theorem}

In \cite{kl}, the Glasner property mainly comes from uniform dilation. In contrast, the Glasner property in Thoerem \ref{t1} should be considered from the dilation in certain directions. Theorem \ref{tl} can also be regarded as a strengthen of the result in \cite{gs}, \cite{sta} and \cite{bq}, which asserts that any infinite $SL(n,\mathbb Z)$-invariant subset of $\mathbb T^n$ is dense. Although our method is not different from \cite{kl}, we use highly nontrivial classification results of stationary measures for the lattice action from \cite{bq}.

We say that, the Lyapunov exponents of an ergodic invariant measure for a $\mathbb Z^k$ action are {\it in general position} if they are all simple and nonzero, and if the Lyapunov hyperplanes are distinct hyperplanes in general position. We say, a subset $A$ {\it is subordinate to} a $\mathbb Z^k$ action, if $A$ is contained in a leaf of a single Lyapunov foliation.

\begin{theorem}\label{tll}
Let $k\ge 2$, and $\alpha$ be an action of $\mathbb Z^k$ on $\mathbb T^{n}$ by ergodic automorphisms. Suppose that the  Lyapunov exponents of the Lebesgue measure are in general position. Let $A$ be an infinite subset which is subordinate to $\alpha$, then $A$ is a Glasner set with respect to $(\mathbb T^n,d_L,\mathbb Z^k)$.
\end{theorem}

In particular, Theorem \ref{tll} applies to Cartan actions on torus. Simply speaking, Cartan action is a $\mathbb R$ split algebraic abelian action on torus with maximal rank.

In \cite{do}, we will make further progress on Glasner property. We will prove, in particular, that parabolic subgroup actions on certain homogeneous spaces have Glasner property.

\subsection{Quantitative density for orbits of infinite subset}

Let's first consider a $\mathbb Z$ action, that is a single transformation. Given the system $(Y,d,T)$, for any infinite subset $A\subset Y$, we can also study the quantitative density of the set $\displaystyle{\bigcup_{k=0}^{n-1}T^kA}$ by the Hausdorff semimetric $d^H$. It will be natural to expect that the density decays much faster than $\frac{1}{n}$, namely $$\inf_n n\cdot d^H\left({\bigcup_{k=0}^{n-1}T^kA},Y\right)=0.$$

Consider arbitrary countable discrete group $G$, and the system $(Y,d,G)$, we can define an increasing family $\{F_n\}$ of subsets of $G$, and consider the density of the set $\displaystyle{\bigcup_{g\in F_n}gA}$. Similarly, one may expect $$\inf_n |F_n|\cdot d^H\left({\bigcup_{g\in F_n}gA},Y\right)=0.$$We will refer this quantitative density problem simply as {\bf Q.D. problem}.

\begin{definition}
Given a system $(Y,d,G)$ and an increasing family $\{F_n\}$ of subsets of $G$, let $A$ be an infinite subset of $Y$. We say, $A$ is a {\bf Q.D. set} with respect to $\{F_n\}$, if $$\inf_n |F_n|\cdot d^H\left({\bigcup_{g\in F_n}gA},Y\right)=0.$$
\end{definition}

\begin{definition}
Given a system $(Y,d,G)$ and an increasing family $\{F_n\}$ of subsets of $G$. If any infinite subset of $Y$ is a Q.D. set with respect to $\{F_n\}$, then we say, $(Y,d,G,\{F_n\})$ has {\bf Q.D. property}.
\end{definition}

Our first result in this direction is about circle rotation.
\begin{theorem}\label{t1}
If $A$ is an infinite subset of $S^1$, then there is a full measure subset $X\subset S^1$, such that for the circle rotation $T_\alpha$ with $\alpha\in X$, $A$ is a Q.D. set with respect to $\{F_n:=[0,n-1]\}$. Namely, we have $$\liminf_nn\cdot d^H\left(\bigcup_{k=0}^{n-1}T_{\alpha}^kA,S^1\right)=0.$$
\end{theorem}

Let $0<\alpha< \beta<1$. A nondegenerate 3 interval exchange transformations (3-IET) $P_{\alpha,\beta}$ is defined by
\begin{eqnarray}
P_{\alpha,\beta}(x)=
\begin{cases}
x+1-\alpha, & 0\le x<\alpha,
\cr 
x+1-\alpha-\beta, &\alpha\le x<\beta,
\cr
x-\beta, & \beta\le x\le 1.
\end{cases}\nonumber
\end{eqnarray}
It is well known \cite{ks} that $P_{\alpha,\beta}$ is the induced map on $[0,1]$ of the translation map $T_{1-\alpha}$ on $[0,1+\beta-\alpha]$. The rotation number of the translation map is $\frac{1-\alpha}{1+\beta-\alpha}$. For 3-IETs, by Theorem \ref{t1} we have

\begin{corollary}\label{t2}
If $A$ is an infinite subset of $I=[0,1]$, then for generic 3 IET $T$, $A$ is a Q.D. set with respect to $\{F_n:=[0,n-1]\}$. Namely, we have $$\liminf_nn\cdot d^H\left(\bigcup_{k=0}^{n-1}T^kA,I\right)=0.$$
\end{corollary}

If the circle rotation is fixed at first, then $(S^1,d_L,\mathbb Z)$ does not have Q.D. property.
\begin{theorem}\label{t3}
For any $\alpha$, there exists an infinite subset $A$ of $S^1$, such that $$\inf_nn\cdot d^H\left(\bigcup_{k=0}^{n-1}T_{\alpha}^kA,S^1\right)>0.$$
\end{theorem}

It seems from the above results that, a $\mathbb Z$ action could rarely have Q.D. property. However, if the acting group is large, then one may expect a lot of actions with this property. We will study this in a forthcoming paper.

\subsection{Remarks about Glasner property and Q.D. property}

It is natural to compare the two properties defined here with certain notion in topological dynamics. We consider a single transformation $T$ on a metric space $Y$ below.

As is well known, topological transitivity asserts that for any {\bf open} subset $U$, any $\epsilon>0$, there is an $n$ such that $T^nU$ is $\epsilon$-dense, namely $\inf_n d^H(T^nU, Y)=0$. Except in the case of discrete topology, an open subset contains infinite many points. Thus by definition, Glasner property is stronger than topological transitivity. Another notion is topological mixing. It is obvious that topological mixing does not imply Glasner property. Indeed, consider any hyperbolic automorphism on torus, it is topological mixing, but never has Glasner property. However, it seems not quite clear whether Glasner property implies topological mixing.

It is still a question whether there exists a dynamical systems ($\mathbb Z$ action) with Glasner property or Q.D. property. We can prove the nonexistence under local connectedness condition, see Theorem \ref{nonexistence} below. In addition, we can not make any assertion on the relation between these two properties. Simply speaking, Glasner property means that any infinite subset will be ``spreaded" to the whole space, and Q.D. property means that the orbit of an infinite subset will be ``spreaded" to the whole space at certain rate. While it is quite possible that these two properties fail for  $\mathbb Z$ action, we emphasize that our study will mostly be focused on ``large" group actions.

\section{Some abstract results}

We start with some abstract results, some of which will be used later.

\begin{proposition}\label{ab1}
Let $(X,T,\mathcal B)$ be a topological dynamical system, $Y$ be a closed subset of $X$, $T_Y$ be the transformation on $Y$ induced from the map $T$. Let $d$ be a metric on $X$, $d_Y$ be the induced metric on $Y$. Denote $d^H_Y$ and $d^H_X$ as the Hausdorff semimetric by $d_Y$ and $d$ respectively. If there exists a constant $M>0$ such that $d^H_Y(A,B)\le M\cdot d^H_X(A,B)$ for any $A,B\subset Y$, then for any $C\subset Y$, $$\inf_nn\cdot d^H_Y\left(\bigcup_{k=0}^{n-1}T_{Y}^kC,Y\right)\le M\inf_nn\cdot d^H_X\left(\bigcup_{k=0}^{n-1}T_{X}^kC,X\right).$$
\end{proposition}

\begin{proof}
The proof is straightforward because $\bigcup_{k=0}^{n-1}T_{X}^kC\cap Y\subset\bigcup_{k=0}^{n-1}T_{Y}^kC$ for any subset $C$.
\end{proof}

Below we show some necessary conditions fo a system generated by one transformation which has Glasner property, although we do not know any example so far.
\begin{proposition}\label{gproperty}
If $(Y,d,T)$ is a topological system with Glasner property, then the following are true.
\begin{itemize}
\item[(1)] $T$ can not be an isometry.
\item[(2)] The orbit of any point is either discrete or dense.
\item[(3)] Every nonatomic ergodic measure of $T$ must have full support.
\item[(4)] If $T$ is uniformly continuous, then $\forall k\neq 0$, $(Y,d,T^k)$ also has Glasner property. In particular, this applies if $Y$ is compact.
\end{itemize}
\end{proposition}
\begin{proof}
The proof of (1), (2) and (3) are straightforward. For (4), fix the $k$. As $T$ is uniformly continuous, then for any $\delta>0$, there is a $\epsilon>0$, such that if $d(x,y)\le \epsilon$, then $\max_{1\le i\le k}d(T^ix,T^iy)\le \delta$. Let $A$ be an infinite subset, then $\forall \epsilon>0$, there exists $n\in\mathbb Z$ such that $T^nA$ is $\epsilon$-dense. Hence there is an $i$, $1\le i\le k$ such that $k|(n+i)$ and $T^{n+i}A$ is $\delta$-dense. Since $\delta$ is chosen arbitrarily, this completes the proof.
\end{proof}

We also give a sufficient condition. Of course, this is far away to be necessary.
\begin{proposition}\label{pro2}
Given a topological system $(Y,d,T)$, consider the infinite (direct) product system $(Y^{\otimes\mathbb N}, D, T^{\otimes\mathbb N})$, here the metric $D$ is defined by $$D(\bar x:=(x_1,x_2,\cdots),\bar y:=(y_1,y_2,\cdots))=\sup_id(x_i,y_i).$$

(1) If for a point $\bar x\in Y^{\otimes\mathbb N}$ such that $x_i\neq x_j$ $\forall i,j\in\mathbb N$, $i\neq j$, the orbit of $\bar x$ under $T^{\otimes\mathbb N}$ is dense, then $A:=\{x_i|i\in\mathbb N\}$ is a Glasner set w.r.t. $(Y,d,T)$.

(2) Assume that for any point $\bar x\in Y^{\otimes\mathbb N}$ such that $x_i\neq x_j$ $\forall i,j\in\mathbb N$, $i\neq j$, the orbit of $\bar x$ under $T^{\otimes\mathbb N}$ is dense, then $(Y,d,T)$ has Glasner property.
\end{proposition}
\begin{proof}
Let $K(Y)$ be the space of subsets of $Y$. Define a map $\pi:Y^{\otimes \mathbb N}\to K(Y)$ as: $$\pi(\bar x)=\{x_i|i\in\mathbb N\}.$$

(1) If the orbit of $\bar x$ is dense, then for any $\epsilon>0$, there exist $k$ such that $\pi((T^{\otimes\mathbb N})^k\bar x)$ is $\epsilon$-dense.

(2) This is essentially the same as (1). We omit the proof.
\end{proof}


\section{$\mathbb Z$ action with Glasner property}

We will show that under local connectedness, no $\mathbb Z$ action  admits Glasner property. We do not know whether it is also the case without the additional assumption.

\begin{theorem}\label{nonexistence}
Let $(Y,d)$ be a compact metric space which is locally connected. Then there does {\bf not} exist a homeomorphism $T$ such that $(Y,d,T)$ has Glasner property.
\end{theorem}

During the proof, we will use several results. We state them first. In \cite{k}, Kato introduced a generalization of expansivity, which is called continuum-wise expansive (cw-expansive). Recall that $T$ is {\it cw-expansive}, if there is $\eta>0$ such that if $C\subset Y$ is connected and $diam (T^n(C))<\eta$ for all $n\in\mathbb Z$ then $C$ is a singleton. Here, $\eta$ is called cw-expansive constant. Combine this with the definition of Glasner property, the following proposition follows.

\begin{proposition}\label{cw}
For a $\mathbb Z$ action, Glasner property implies cw-expansivity.
\end{proposition}

For the homeomorphism $T$, define the $\epsilon$-local stable set of a point $x$ in $Y$ as the set $$W^s_\epsilon(x)=\{y\in Y:d(T^nx,T^ny)\le\epsilon,\;\forall n\ge 0\}.$$
Define similarly the $\epsilon$-local unstable set $W^u_\epsilon(x)$. Denoting by $CW_\epsilon^\sigma(x)$ the connected component of $x$ in the set $W_\epsilon^\sigma(x)$ for $\sigma= s,u$. The following useful result guarantees the existence of non trivial local stable/unstable sets.

\begin{theorem}[\cite{ro}]\label{jan}
If $T$ is a cw-expansive homeomorphism on a locally connected compact metric space $Y$, then for any $\epsilon > 0$ there exists $\delta > 0$ such that$$\inf_{x\in Y'}diam(CW_\epsilon^s(x))\ge\delta,\;\inf_{x\in Y'}diam(CW_\epsilon^u(x))\ge\delta.$$Here $Y'$ is the set of accumulation points of $Y$.
\end{theorem}

Now we are ready to give the proof.
\begin{proof}[Proof of Theorem \ref{nonexistence}]
By contradiction, assume there is a homeomorphism $T$ such that $(Y,d,T)$ has Glasner property. By Proposition \ref{cw}, $T$ is cw-expansive. In fact, it is cw-expansive with cw-expansive constant $\eta$ as long as $\eta<\sup{d(x,y)}$.

Fix $\epsilon>0$ small enough, by Theorem \ref{jan}, there is a $\delta>0$, such that for any $x\in Y'$, $diam(CW_\epsilon^s(x))\ge\delta$. Let's regard $4\delta$ as the cw-expansive constant.

It is known \cite[Corollary 2.4]{k} that there is $N> 0$ such that if $diam(CW_\epsilon^s(x)) = \delta$, then $$diam(T^{-N}(CW_\epsilon^s(x))) >4 \delta.$$Let $g=T^{-N}$, by Proposition \ref{gproperty}, $(Y,d,g)$ also has Glasner property. 

Now choose an $x\in Y'$, let $C_0=CW^s_\epsilon(x)$ and $U=\{y\in Y:d(x,y)\le\delta/4\}$. Since $diam(C_0)\ge \delta$, then $diam(g(C_0))>3\delta$. Hence we can choose two disjoint connected component $C_1^1$ and $C_1^2$ from $g(C_0)-U$ such that $diam(C_1^{1/2})=\delta$. Replace $C_0$ by $C_1^1$ and also by $C_1^2$, we can get four connected component $C_2^i$ from $g(C_0)-U$ for $1\le i\le 4$ such that $diam(C_2^{i})=\delta$. Continuing in this way, we can get at the $k$th step, $2^k$ connected component $C_k^i$ for $1\le i\le 2^k$ such that $diam(C_k^{i})=\delta$. Let $$\hat C=\bigcap_{k=0}^\infty \bigcup_{1\le i\le 2^k}g^{-k}C_k^i.$$It is obvious from the construction, that $\hat C$ is Cantor set, and hence contains infinite many points.

However, since $\hat C\subset C_0\subset CW_\epsilon^s(x)$, $g^n(\hat C)\subset g^n(C_0)\subset g^n(CW_\epsilon^s(x))\subset B(g^nx,\epsilon)$ for $n\le 0$, and by the construction, $g^{n}(\hat C)\cap U=\emptyset$ for $n>0$. Therefore $$\inf_{n\in\mathbb Z}d^H(g^n(C),Y)\ge \delta/4>0.$$This is a contradiction to the Glasner property for $g$.
\end{proof}

\section{Proofs}

\subsection{Proof of Theorem \ref{tl}}

Theorem \ref{tl} follows easily from the following quantitative result, whose proof is carried over through this subsection. The idea of the proof comes originally from \cite{ap}, which is generalized for much more broader cases in \cite{kl}.

\begin{theorem}\label{new1}
For any set $A$ of $k$ distinct points in $\mathbb T^n$, if there is no $\gamma\in \Gamma$ such that $\gamma A$ is $\epsilon$-dense, then $$k\le  \frac{C_n}{\epsilon^{C_{n,\epsilon}}},$$provided $\epsilon>0$ small enough.
\end{theorem}

Let $\{x_1,x_2,\cdots,x_k\}$ be a set of $k$ distinct points in $\mathbb T^n$. Define $$h_m:=\#\{(i,j)|1\le i,j\le k\text{ such that }m(x_i-x_j)\in\mathbb Z^n\},$$let $H_m:=h_1+h_2+\cdots+h_m$. The estimates of $h_m$ and $H_m$ already appear in the work \cite{ap} and \cite{kl}. Using the same idea as the proof of Proposition 1 in \cite{kl}, we can get

\begin{proposition}\label{pro1}
For any positive integers $k$ and $m$, $H_m\le km^{n+1}$.
\end{proposition} 

\begin{lemma}\label{cor1}
Let $r>1$. If $s_2,s_3,\cdots$ is a sequence of nonnegative integers such that $S_b=s_2+s_3+\cdots+s_b\le H_b$, and $S_b\le k^2$. Then $$\sum_{b=2}^\infty s_bb^{-r}\le C_{n,r} k^{2-r/(n+1)}.$$
\end{lemma}
\begin{proof}
The proof is carried over similarly as that of \cite[Corollary 1]{kl}. Let $S_1=0$, note that 
\begin{equation}\begin{split}
\sum_{b=2}^\infty s_bb^{-r}&=\sum_{b=2}^\infty (S_b-S_{b-1})b^{-r}=\sum_{b=2}^\infty S_b(b^{-r}-(b+1)^{-r})\\&\le \sum_{b=2}^{[k^{1/(n+1)}]}S_b(b^{-r}-(b+1)^{-r})+k^2\cdot k^{-r/(n+1)}\\&\le \sum_{b=2}^{[k^{1/(n+1)}]}H_bb^{-r-1}+k^{2-r/(n+1)}\\&\le\sum_{b=2}^{[k^{1/(n+1)}]}kb^{n+1}b^{-r-1}+k^{2-r/(n+1)}\le Ck\cdot k^{\frac{n-r+1}{n+1}}+k^{2-r/(n+1)}\\&=Ck^{2-r/(n+1)}.
\end{split}\nonumber\end{equation}
\end{proof}

We will appy a special case of a general theorem about stationary measure for discrete group actions proven by Benoist and Quint. We reformulate it for our case. Let $\nu$ be a probability measure supported on a finite set of generators of $\Gamma$.
\begin{theorem}[\cite{bq}]\label{limit}
For any $\phi\in C_c(\mathbb T^n)$ and $x\in\mathbb T^n$, then
$$\frac{1}{N}\sum_{\gamma\in\Gamma}\sum_{k=0}^{N-1}\nu^{*k}(\gamma)\phi(\gamma^{-1} x)\to \int_{\mathbb T^n}\phi d\mu_x$$as $N\to\infty$. Here $\nu^{*k}$ means convolution of $\nu$ $k$ times.
\end{theorem}

\begin{remark}
By results from \cite{gs}, \cite{sta} and \cite{bq}, either $\Gamma x$ is discrete or it is dense. When $\overline{\Gamma x}$ is discrete,  then $x\in\mathbb Q^n$, and $\mu_x$ is an atomic measure depending on $x$ which is invariant under $\Gamma$ action. When $\Gamma x$ is dense, then $\mu_x$ is the Lebesgue measure.
\end{remark}

Let $\mathbf m=(m_1,\cdots,m_n)\in\mathbb Z^n$ be a nonzero integer vector, $q$ be a positive integer. Let $$c_q(\mathbf m):=\sum_{\mathbf k}e\left(\frac{\langle\mathbf m,\mathbf k\rangle}{q}\right),$$here $\langle\cdot,\cdot\rangle$ is the usual inner product and $e(z)=e^{2\pi iz}$. The sum is taken over all $\mathbf k=(k_1,\cdots,k_n)$ such that $1\le k_i\le q$ and $gcd(k_1,\cdots,k_n,q)=1$. Let $\phi_q$ be the number of such $\mathbf k$. It is easy to see that $\phi_q$ grows like $(q/\log q)^n$. Therefore there is a constant $C_0>1$ such that $\phi_q\ge C_0q^{n-1}$ for any $q\ge 1$. We would like to have an upper bound for $c_q(\mathbf m)$. Notice that when $n=1$, $c_q(\mathbf m)$ is the famous Ramanujan sum. 
\begin{lemma}\label{ram}
For $\mathbf m\neq\mathbf 0$, $|c_q(\mathbf m)|\le \left(gcd(m_1,\cdots,m_n)\right)^n$.
\end{lemma}
\begin{proof}
By \cite{ph}, $c_q(\mathbf m)=\prod_{p^r\parallel q}c_{p^r}(\mathbf m)$ and $$c_{p^r}(\mathbf m)=\begin{cases}
p^{(r - 1)n} (p^n - 1) & \text{if }p^r \mid gcd(m_1,\cdots,m_n), \\
-p^{(r - 1)n} & \text{if }p^{r - 1} \parallel gcd(m_1,\cdots,m_n), \\
-1 & \text{if }p^{r - 1} \nmid gcd(m_1,\cdots,m_n).
\end{cases}$$From here, the lemma follows easily.
\end{proof}

We will also need the following fact about ``bump" functions.
\begin{lemma}\label{newl}
For $0<\epsilon<1$, there exists a nonnegative function $g_\epsilon:\mathbb T^n\to \mathbb R$, such that

\begin{itemize}
\item[(1)] $\int_{\mathbb T^n}g_\epsilon(x) dx=1$,
\item[(2)] $g_\epsilon(x)=0$ if $d_L(x,\bf e)\ge \epsilon$, here $d_L$ is the standard metric on $\mathbb T^n$, and $\mathbf e\in\mathbb T^n$ is the identity,
\item[(3)] there exists a positive constant $C_1$ such that, for any $\mathbf m\in \mathbb Z^n$, the Fourier coefficient $\hat g_\epsilon(\mathbf m)=\int_{\mathbb T^n}g_\epsilon(x)e_{\mathbf m}(x)dx$ satisfies$$|\hat g_\epsilon(\mathbf m)|\le C_1e^{-\sqrt{\epsilon |\mathbf m|}}.$$
\end{itemize}Here $|\mathbf m|=\sum|m_i|$ is a norm on $\mathbb Z^n$.
\end{lemma}
\begin{proof}
This follows from the circle version, which appears as Lemma 6.2 in \cite{ap}. Indeed, let $\tilde g_\epsilon$ be the function obtain from Lemma 6.2 \cite{ap}, for $x=(x_1,\cdots,x_n)$, let $g_\epsilon(x)=\prod_{i=1}^n\tilde g_\epsilon(x_i)$, then one can easily check (1), (2) and (3). Let's remark that $$\int_{\mathbb T^n} g_\epsilon^2(x)dx=\prod_{i=1}^n\int_{\mathbb T}\tilde g_\epsilon^2(x_i)dx_i\le\frac{C}{\epsilon^n}.$$
\end{proof}

\begin{proof}[Proof of Theorem \ref{new1}]
We use the idea of the proof of Proposition 6.1 in \cite{ap}.

For $\epsilon>0$ small, let $g_\epsilon$ be a function from Lemma \ref{newl}. Let $\{x_1,x_2,\cdots,x_k\}$ be a set of $k$ distinct points in $\mathbb T^n$, and suppose there is no $\gamma\in \Gamma$ such that $\gamma A$ is $\epsilon$-dense. Then for each $\gamma$, there is a point $x_\gamma\in\mathbb T^n$ such that $g_\epsilon(\gamma x_i-x_\gamma)=0$ for any $i$.

Hence we have, for any $N>0$,
\begin{equation}\label{ee1}
0=\frac{1}{N}\sum_{\gamma\in\Gamma}\sum_{j=0}^{N-1}\nu^{*j}(\gamma)\sum_{i=1}^kg_\epsilon(\gamma x_i-x_\gamma)=\frac{1}{N}\sum_{\gamma\in\Gamma}\sum_{j=0}^{N-1}\nu^{*j}(\gamma)\sum_{i=1}^k\sum_{\mathbf m\in\mathbb Z^n}\hat g_\epsilon(\mathbf m)e_{\mathbf m}(\gamma^{-1} x_i-x_\gamma).
\end{equation}

Since $\hat g_\epsilon(\mathbf 0)=\int_{\mathbb T^n}g_\epsilon(x) dx=1$, then from (\ref{ee1}),

\begin{equation}\label{ee2}
k\le \left|\frac{1}{N}\sum_{\gamma\in\Gamma}\sum_{j=0}^{N-1}\nu^{*j}(\gamma)\sum_{i=1}^k\sum_{\mathbf m\in\mathbb Z^n\backslash\mathbf 0,|\mathbf m|\le M}\hat g_\epsilon(\mathbf m)e_{\mathbf m}(\gamma^{-1} x_i-x_\gamma)\right|+k\sum_{\mathbf m\in\mathbb Z^n,|\mathbf m|> M}\left|\hat g_\epsilon(\mathbf m)\right|.
\end{equation}

Now by (3) of Lemma \ref{newl}, 
\begin{eqnarray}
\sum_{\mathbf m\in\mathbb Z^n,|\mathbf m|> M}\left|\hat g_\epsilon(\mathbf m)\right| &\le& C_1\sum_{\ell=M+1}^\infty C_{\ell+1}^{n-1}e^{-\sqrt{\ell\epsilon}}\le 2C_1\sum_{\ell=M+1}^\infty \ell^{n-1}e^{-\sqrt{\ell\epsilon}}\nonumber\\
&\le& \frac{2C_1}{\epsilon^{n-1}}\int_{M\epsilon}^\infty x^{n-1}e^{-\sqrt{x}}dx= \frac{4C_1}{\epsilon^{n-1}}\int_{\sqrt{M\epsilon}}^\infty u^{2n-1}e^{-u}du.
\end{eqnarray}

As the integral $\int_{0}^\infty u^{2n-1}e^{-u}du=\Gamma(2n)<\infty$, there exists a great enough number $a$ (depends on $\epsilon$) such that for $M:=\frac{1}{\epsilon^a}$, $$\int_{\sqrt{M\epsilon}}^\infty u^{2n-1}e^{-u}du\le \frac{\epsilon^{n-1}}{8C_1},$$and hence$$\sum_{\mathbf m\in\mathbb Z^n,|\mathbf m|> M}\left|\hat g_\epsilon(\mathbf m)\right|\le \frac{1}{2}.$$

Hence from (\ref{ee2}), we have for $M\ge M_0$,
\begin{equation}\label{ee4}
\frac{k}{2}\le \left|\frac{1}{N}\sum_{\gamma\in\Gamma}\sum_{j=0}^{N-1}\nu^{*j}(\gamma)\sum_{i=1}^k\sum_{\mathbf m\in\mathbb Z^n\backslash\mathbf 0,|\mathbf m|\le M}\hat g_\epsilon(\mathbf m)e_{\mathbf m}(\gamma^{-1} x_i-x_\gamma)\right|.
\end{equation}
Now square both sides of the above inequality, then apply Cauchy-Schwartz inequality twice,
\begin{equation}\label{ee5}\begin{split}
\frac{k^2}{4}\le &\left|\sum_{\gamma\in\Gamma}\frac{1}{N}\sum_{j=0}^{N-1}\nu^{*j}(\gamma)\sum_{i=1}^k\sum_{\mathbf m\in\mathbb Z^n\backslash\mathbf 0,|\mathbf m|\le M}\hat g_\epsilon(\mathbf m)e_{\mathbf m}(\gamma^{-1} x_i-x_\gamma)\right|^2\\
= &\left|\sum_{\gamma\in\Gamma}\sqrt{\frac{1}{N}\sum_{j=0}^{N-1}\nu^{*j}(\gamma)}\cdot \sqrt{\frac{1}{N}\sum_{j=0}^{N-1}\nu^{*j}(\gamma)}\sum_{i=1}^k\sum_{\mathbf m\in\mathbb Z^n\backslash\mathbf 0,|\mathbf m|\le M}\hat g_\epsilon(\mathbf m)e_{\mathbf m}(\gamma^{-1} x_i-x_\gamma)\right|^2\\
\le &\left(\sum_{\gamma\in\Gamma}\frac{1}{N}\sum_{j=0}^{N-1}\nu^{*j}(\gamma)\right)\cdot\\&\left(\sum_{\gamma\in\Gamma}\frac{1}{N}\sum_{j=0}^{N-1}\nu^{*j}(\gamma)\left|\sum_{\mathbf m\in\mathbb Z^n\backslash\mathbf 0,|\mathbf m|\le M}\hat g_\epsilon(\mathbf m)\sum_{i=1}^ke_{\mathbf m}(\gamma^{-1} x_i-x_\gamma)\right|^2\right)\\
\le &\sum_{\gamma\in\Gamma}\frac{1}{N}\sum_{j=0}^{N-1}\nu^{*j}(\gamma)\left(\sum_{\mathbf m\in\mathbb Z^n\backslash\mathbf 0,|\mathbf m|\le M}|\hat g_\epsilon(\mathbf m)|^2\right)\cdot\\&\left(\sum_{\mathbf m\in\mathbb Z^n\backslash\mathbf 0,|\mathbf m|\le M}\left|\sum_{i=1}^ke_{\mathbf m}(\gamma^{-1} x_i-x_\gamma)\right|^2\right)\\
= &\left(\sum_{\mathbf m\neq \mathbf0,|\mathbf m|\le M}|\hat g_\epsilon(\mathbf m)|^2\right)\cdot\\&\left(\sum_{\mathbf m\in\mathbb Z^n\backslash\mathbf 0,|\mathbf m|\le M}\sum_{1\le i,\ell\le k}\sum_{\gamma\in\Gamma}\frac{1}{N}\sum_{j=0}^{N-1}\nu^{*j}(\gamma)e_{\mathbf m}(\gamma^{-1} (x_i-x_\ell))\right).
\end{split}\end{equation}

By Bessel's inequality and the construction in Lemma \ref{newl}, \begin{equation}\label{ee6}
\left(\sum_{\mathbf m\in\mathbb Z^n\backslash\mathbf 0,|\mathbf m|\le M}|\hat g_\epsilon(\mathbf m)|^2\right)\le \frac{C_2}{\epsilon^n}.
\end{equation}

Combining (\ref{ee5}) and (\ref{ee6}), we obtain
\begin{equation}\label{ee7}
k^2\le \frac{4C_2}{\epsilon^n}\left(\sum_{\mathbf m\in\mathbb Z^n\backslash\mathbf 0,|\mathbf m|\le M}\sum_{1\le i,\ell\le k}\lim_{N\to\infty}\sum_{\gamma\in\Gamma}\frac{1}{N}\sum_{j=0}^{N-1}\nu^{*j}(\gamma)e_{\mathbf m}(\gamma^{-1} (x_i-x_\ell))\right).
\end{equation}

By Theorem \ref{limit} and the remark after it, as $N\to\infty$, $\sum_{\gamma\in\Gamma}\frac{1}{N}\sum_{j=0}^{N-1}\nu^{*j}(\gamma)e_{\mathbf m}(\gamma^{-1} (x_i-x_\ell))$ contributes to the right hand sum in (\ref{ee7}) {\bf only} when $(x_i-x_j)\in\mathbb Q^n/\mathbb Z^n$. When $(x_i-x_j)\in\mathbb Q^n/\mathbb Z^n\backslash\{\mathbf 0\}$, let $q_{i,\ell}$ be the least positive integer that $q_{i,\ell}(x_i-x_\ell)\in \mathbb Z^n$. Hence by Theorem \ref{limit}, $$\sum_{\gamma\in\Gamma}\frac{1}{N}\sum_{j=0}^{N-1}\nu^{*j}(\gamma)e_{\mathbf m}(\gamma^{-1} (x_i-x_\ell))\to \frac{1}{\phi_{q_{i,\ell}}}\sum_{\mathbf k}e\left(\frac{\langle\mathbf m,\mathbf k\rangle}{q_{i,\ell}}\right)=\frac{1}{\phi_{q_{i,\ell}}}c_{q_{i,\ell}}(\mathbf m).$$ By Lemma \ref{ram}, $c_{q_{i,\ell}}(\mathbf m)\le M^n$ because $|\mathbf m|\le M$. Combine this with $\phi_q\ge C_0 q^{n-1}$, we have $$\left|\sum_{\gamma\in\Gamma}\frac{1}{N}\sum_{j=0}^{N-1}\nu^{*j}(\gamma)e_{\mathbf m}(\gamma^{-1} (x_i-x_\ell))\right|\le \frac{M^n}{C_0 q_{i,\ell}^{n-1}}$$ as $N\to\infty$.

Recall that $M=\frac{1}{\epsilon^a}$, by Proposition \ref{pro1} and Lemma \ref{cor1}, let $N\to\infty$
\begin{eqnarray}\label{ee8}
& &\sum_{\mathbf m\in\mathbb Z^n\backslash\mathbf 0,|\mathbf m|\le M}\sum_{1\le i,\ell\le k}\left|\sum_{\gamma\in\Gamma}\frac{1}{N}\sum_{j=0}^{N-1}\nu^{*j}(\gamma)e_{\mathbf m}(\gamma^{-1} (x_i-x_\ell))\right|\nonumber\\ &\le & \left( 2^nM^n k+ \sum_{\mathbf m\in\mathbb Z^n\backslash\mathbf 0,|\mathbf m|\le M}\sum_{1\le i\neq \ell\le k}\left|\sum_{\gamma\in\Gamma}\frac{1}{N}\sum_{j=0}^{N-1}\nu^{*j}(\gamma)e_{\mathbf m}(\gamma^{-1} (x_i-x_\ell))\right|\right)\nonumber\\ &\le & \left( 2^nM^n k+ C_{n}M^n\sum_{q=1}^\infty h_qq^{-n+1}\right)\nonumber\\
&\le & \left( 2^nM^n k+ C_{n}M^nk^{2-(n-1)/(n+1)}\right)\nonumber\\
&\le & C_{n}M^nk^{1+1/(n+1)}\nonumber\\
&=&\frac{C_{n}k^{1+1/(n+1)}}{\epsilon^{na}}.
\end{eqnarray}

Combine (\ref{ee7}) and (\ref{ee8}), we obtain $$k\le \frac{C(n)}{\epsilon^{2(a+1)n}}.$$
This finishes the proof.
\end{proof}

A direct corollary of Theorem \ref{tl} is the following concerning to finite index subgroups of $SL(n,\mathbb Z)$. 
\begin{corollary}
Let $n\ge 2$, and $\Gamma$ be a finite index subgroup of $SL(n,\mathbb Z)$. Then the system $(\mathbb T^n,d_L,\Gamma)$ has Glasner property.
\end{corollary}

\begin{remark}
We will extend the above results in \cite{do} via a new method, proving that the system $(\mathbb T^n,d_L,\Gamma)$ has Glasner property if $\Gamma<SL(n,\mathbb Z)$ is Zariski dense in $SL(n,\mathbb R)$. Unfortunately, we could not obtain this here.
\end{remark}

\subsection{Proof of Theorem \ref{tll}}

We will use some of the results and notations from the previous subsection, and modify the previous argument to prove Theorem \ref{tll}. We start with the following useful property of algebraic abelian actions.

\begin{proposition}\label{prol}
Let $k\ge 2$, and $\alpha$ be an action of $\mathbb Z^k$ on $\mathbb T^{n}$ by ergodic automorphisms. Suppose that the  Lyapunov exponents of the Lebesgue measure are in general position. Then for any Lyapunov exponent $\chi$, $\{\chi(\mathbf n)|\mathbf n\in \mathbb Z^k\}$ is dense in $\mathbb R$.
\end{proposition}

\begin{proof}
Since $\chi:\mathbb Z^k\to \mathbb R$ is a linear functional, it suffices to prove that for any $\epsilon>0$, there exists an $\mathbf n\in \mathbb Z^k$ such that $|\chi(\mathbf n)|\le \epsilon$. Now since the Lyapunov exponents are in general position, each Lyapunov hyperplane (a $k-1$ dimensional hyperplane in $\mathbb R^k$) has no intersection with $\mathbb Z^k$. Therefore, for any $\epsilon>0$, there always exists an element $\mathbf n\in\mathbb Z^k$, which is very close to the hyperplane corresponding to $\chi$, satisfying that $|\chi(\mathbf n)|\le \epsilon$. This finishes the proof.
\end{proof}

We will prove the following quantitative version, which indicates Theorem \ref{tll}.
\begin{theorem}
Let $k\ge 2$, and $\alpha$ be an action of $\mathbb Z^k$ on $\mathbb T^{n}$ by ergodic automorphisms. Suppose that the  Lyapunov exponents of the Lebesgue measure are in general position. Let $A$ be a finite subset which is subordinate to $\alpha$. If for an $\epsilon>0$ small enough, there exists no $\mathbf n\in \mathbb Z^k$ such that $\alpha(\mathbf n)A$ is $\epsilon$-dense, then $\#A\le\frac{C}{\epsilon^{C_{\epsilon,n}}}$.
\end{theorem}

\begin{proof}
Let $A=\{x_1,\cdots,x_\ell\}$ with $\ell=\#A$. Assume that $A$ lies in the leaf of the foliation corresponding to the Lyapunov exponent $\chi$. Let $p_{j,r}\in\mathbb R$ be such that $x_j-x_r=p_{j,r}\mathbf v\mod\mathbb Z^n$, here $\mathbf v$ is a fixed vector on $\mathbb T^n$ generating the leaf where $A$ lies in.

First, fix an integer $N>1$. By Proposition \ref{prol}, we can choose $\mathbf n_i\in\mathbb Z^k$ for $i\in\mathbb N$, so that $e^{\chi(\mathbf n_i)}$ is uniformly distributed in $[1,N]$. Since there exists no $\mathbf n\in \mathbb Z^k$ such that $\alpha(\mathbf n)A$ is $\epsilon$-dense, then for each $i$, there is a $y_i\in\mathbb T^n$ such that $g_\epsilon(\alpha(\mathbf n_i)x_j-y_i)=0$ for any $j$.

Therefore for any $L>1$,
$$0=\frac{1}{L}\sum_{i=1}^{L}\sum_{j=1}^\ell g_\epsilon(\alpha(\mathbf n_i)x_j-y_i)=\frac{1}{L}\sum_{i=1}^{L}\sum_{j=1}^\ell\sum_{\mathbf m\in \mathbb Z^n} \hat g_\epsilon(\mathbf m)e_{\mathbf m}(\alpha(\mathbf n_i)x_j-y_i).$$
From here, analyzing in the same way as to obtain Equation (\ref{ee7}) in the proof of Theorem \ref{new1}, we can get
$$\ell^2\le \frac{C_2}{\epsilon^n}\left(\sum_{\mathbf m\in \mathbb Z^n\backslash\mathbf 0,|\mathbf m|\le M}\sum_{j=1}^\ell\sum_{r=1}^\ell\frac{1}{L}\sum_{i=1}^{L}e_{\mathbf m}(\alpha(\mathbf n_i)(x_j-x_r))\right).$$

Observe that $\alpha(\mathbf n_i)(x_j-x_r)=e^{\chi(\mathbf n_i)}(x_j-x_r)\mod \mathbb Z^n=e^{\chi(\mathbf n_i)}p_{j,r}\mathbf v\mod \mathbb Z^n$, hence $e_{\mathbf m}(\alpha(\mathbf n_i)(x_j-x_r))=e_m(e^{\chi(\mathbf n_i)}p_{j,r}\mathbf v)$. By the choice of $\mathbf n_i$, let $L\to\infty$, thus we have
\begin{eqnarray}\label{n9}
\ell^2&\le& \frac{C_2}{\epsilon^n}\left(\sum_{\mathbf m\in \mathbb Z^n\backslash\mathbf 0,|\mathbf m|\le M}\sum_{j=1}^\ell\sum_{r=1}^\ell\frac{1}{N-1}\int_{1}^{N}e_{\mathbf m}(tp_{j,r}\mathbf v)dt\right)\nonumber\\
&=& \frac{C_2}{\epsilon^n}\left(2^nM^n\ell+\sum_{\mathbf m\in \mathbb Z^n\backslash\mathbf 0,|\mathbf m|\le M}\sum_{1\le j\neq r\le \ell}\frac{1}{N-1}\int_{1}^{N}e_{\mathbf m}(tp_{j,r}\mathbf v)dt\right).
\end{eqnarray}
Notice that for $j\neq r$, $p_{j,r}\neq 0$, the flow on $\mathbb T^n$ generated by the vector $p_{j,r}\mathbf v$ is uniquely ergodic, thus $$\lim_{N\to\infty}\frac{1}{N-1}\int_{1}^{N}e_{\mathbf m}(tp_{j,r}\mathbf v)dt=\int_{\mathbb T^n} e_{\mathbf m}(x)dx=0.$$Therefore let $N\to\infty$ in (\ref{n9}), we have $$\ell^2\le\frac{C_3}{\epsilon^n}M^n\ell=\frac{C_3\ell}{\epsilon^{(a+1)n}}.$$ Hence $\ell\le\frac{C_3}{\epsilon^{(a+1)n}}$.
\end{proof}

\subsection{Proofs of the other results}

We are going to use the subsequent two results, as they play an important role in the proof. The first one is a strengthen of the aforementioned Glasner's result.

\begin{theorem}[Theorem 1.3, \cite{bp}]\label{c1}
For any infinite subset $X\subset S^1$, there exists a sequence $\{n_i\}$ with density 1 in $\mathbb N$, such that $E_{n_i}X\to S^1$ in the Hausdorff semimetric. Here $E_{n}x=nx\mod 1$.
\end{theorem}

Namely, there exists a sequence $\{n_i\}$ with density 1 in $\mathbb N$, such that$$\lim_{n_i}n_i\cdot d^H\left(\bigcup_{k=0}^{n_i-1}T_{\frac{1}{n_i}}^kX,S^1\right)=0.$$This is because if $E_nX$ is $\epsilon$-dense, then its preimage $E_n^{-1}(E_nX)$ is $\frac{\epsilon}{n}$-dense, and the fact$$E_n^{-1}(E_nX)=\bigcup_kT_{\frac{1}{n}}^kX.$$

\begin{theorem}[Theorem III, \cite{ds}]\label{c2}
Let $Q$ be an increasing sequence of integers with a positive lower density. Let $a_1,a_2,\cdots$ be a sequence of positive numbers such that $\sum_qa_q=\infty$, and for some real number $c$, $\frac{a_q}{q^c}$ is a decreasing function of $q$. Then for almost all $x$, there exist arbitrarily many relative prime $p$ and $q$, such that $$\left|x-\frac{p}{q}\right|\le \frac{a_q}{q}, q\in Q.$$
\end{theorem}

We are now ready to give the proof.

\begin{proof}[Proof of Theorem \ref{t1}]
Let $X$ be an arbitrary infinite subset of $S^1$, then by Theorem \ref{c1}, there is an increasing sequence $\{n_i\}$ of natural numbers with density 1 in $\mathbb N$, such that $\mathbb N$, such that
\begin{equation}\label{e1}
\lim_{n_i}n_i\cdot d^H\left(\bigcup_{k=0}^{n_i-1}T_{\frac{1}{n_i}}^kX,S^1\right)=0.
\end{equation}
Now apply Theorem \ref{c2} with $Q=\{n_i\}$ and $a_q=\frac{1}{q\log q}$, we have that there is a subset $U\subset S^1$ of full measure, such that for any $\alpha\in U$, there exist infinitely many $n_i\in Q$, such that 
\begin{equation}\label{e2}
\left|\alpha-\frac{p}{n_i}\right|\le \frac{1}{n_i^2\log n_i}\text{ for some }p, (p,n_i)=1.\end{equation}

Now for any $\alpha\in U$, we claim that $$\liminf_nn\cdot d^H\left(\bigcup_{k=0}^{n-1}T_{\alpha}^kX,S^1\right)=0.$$Let $k_i$ be the subsequence of $Q$ such that (\ref{e2}) holds for $\alpha$. Then we have
\begin{eqnarray}
k_id^H\left(\bigcup_{j=0}^{k_i-1}T_{\alpha}^jX,S^1\right)&\le& k_id^H\left(\bigcup_{j=0}^{k_i-1}T_{\alpha}^jX,\bigcup_{j=0}^{k_i-1}T_{\frac{p}{k_i}}^jX\right)+k_id^H\left(\bigcup_{j=0}^{k_i-1}T_{\frac{p}{k_i}}^jX,S^1\right)\nonumber \\
&\le &k_i^2\left|\alpha-\frac{p}{k_i}\right|+k_id^H\left(\bigcup_{j=0}^{k_i-1}T_{\frac{p}{k_i}}^jX,S^1\right)\nonumber \\
&\le & \frac{1}{\log k_i}+k_id^H\left(\bigcup_{j=0}^{k_i-1}T_{\frac{p}{k_i}}^jX,S^1\right).\nonumber 
\end{eqnarray}
Combine the above inequality with (\ref{e1}), the claim follows.
\end{proof}

We can use the same idea to prove the following
\begin{theorem}\label{tt1}
If $A_i$ is an infinite subset of $S^1$ for $i=1,2$, then for almost every $\alpha$, $$\liminf_nn\cdot d^H\left(\bigcup_{k=0}^{n-1}T_{\alpha}^kA_1,\bigcup_{k=0}^{n-1}T_{\alpha}^kA_2\right)=0.$$
\end{theorem}
\begin{proof}
Note that, by Theorem \ref{c1}, there is an increasing sequence $\{n_i\}$ of natural numbers with density 1 in $\mathbb N$, such that $\mathbb N$, such that (\ref{e1}) holds for both $A_1$ and $A_2$. Using the triangle inequality, the rest of the argument will be essentially the same as the proof of Theorem \ref{t1}. We omit the details here.
\end{proof}

\begin{proof}[Proof of Corollary \ref{t2}]
The proof follows from Proposition \ref{ab1} and the fact in \cite{ks}. We give the details here.

Let $X$ be an arbitrary infinite subset of $I=[0,1]$. We can also think of $X$ as a subset of $[0,1+\beta-\alpha]$. Thus by Theorem \ref{t1}, for fixed $\beta-\alpha:=\ell$, and generic $\alpha\in (0,1-\ell)$ (Namely, $\frac{1-\alpha}{1+\beta-\alpha}$ satisfies a generic condition), then $$\liminf_nn\cdot d^H\left(\bigcup_{k=0}^{n-1}T_{1-\alpha}^kX,[0,1+\beta-\alpha]\right)=0.$$
Now note that $$\left(\bigcup_{k=0}^{n-1}T_{1-\alpha}^kX\right)\cap I \subset \bigcup_{k=0}^{n-1}P_{\alpha,\beta}^kX,$$
hence by Proposition \ref{ab1},
$$\liminf_nn\cdot d^H\left(\bigcup_{k=0}^{n-1}P_{\alpha,\beta}^kX,I\right)=0.$$
By applying Fubini's Theorem on the pair $(\alpha,\ell)$, the claim follows.
\end{proof}

\begin{proof}[Proof of Theorem \ref{t3}]
If $\alpha\in\mathbb Q$, then it suffices to let $X=[0,\epsilon]$ for $\epsilon>0$ small enough. 

Assume that $\alpha\notin \mathbb Q$. Let the integers $a_k$ be the coefficients of the continued fraction of $\alpha$, and $\frac{p_k}{q_k}$ be the partial convergence. Let $||q\alpha||=dist(q\alpha,\mathbb Z)$, $|x|$ be the fraction part of $x$. Then 
\begin{itemize}
\item[(1)] $\frac{1}{q_{k+1}+q_k}<||q_k\alpha||<\frac{1}{a_{k+1}q_k}$, 
\item[(2)] $d^H(\bigcup_{i=0}^{q_{k+1}-1}\{|i\alpha|\},S^1)\ge ||q_k\alpha||$.
\end{itemize}
From (1) and (2), it is easy to see that 
\begin{equation}\label{e3}
\liminf_nn\cdot d^H\left(\bigcup_{k=0}^n\{T^k(0)\},S^1\right)>0.
\end{equation}

Choose an increasing subsequence $\{n_k\}$ of $\mathbb N$, such that $\log\log\log(q_{n_{k+1}})\ge q_{n_k}$, and $0<|q_{n_{k+1}}\alpha|<|q_{n_k}\alpha|$. Let $X=\{|q_{n_k}\alpha|:k\in \mathbb N\}\cup\{0\}$. We claim that $$\liminf_nn\cdot d^H\left(\bigcup_{k=0}^{n-1}T_{\alpha}^kX,S^1\right)>0.$$

Let $n$ be any integer, and $k$ such that $q_{n_k}-q_{n_{k-1}}\le n\le q_{n_{k+1}}-q_{n_k}$. Then $$\bigcup_{i=0}^{n-1}T_{\alpha}^iX\subset \bigcup_{i=0}^{q_{n_k}+n}\{T_{\alpha}^i(0)\}\cup\bigcup_{i=0}^{n-1}\{T_{\alpha}^{i+q_{n_{k+1}}}(0)\}\cup\bigcup_{j=k+2}^\infty\bigcup_{i=0}^{n-1}\{T_{\alpha}^{i+q_{n_{j}}}(0)\}.$$
Now by the choice of $n_k$, when $j\ge k+2$, $\bigcup_{i=0}^{n-1}\{T_{\alpha}^{i+q_{n_{j}}}(0)\}$ is at least $|q_{n_{k+2}}\alpha|$ close to $\bigcup_{i=0}^{n-1}\{T_{\alpha}^{i}(0)\}$ because $|q_{n_j}\alpha|\ll|q_{n_{k+1}}\alpha|$. Hence there is a positive constant $c$, such that 
\begin{eqnarray}
d^H\left(\bigcup_{i=0}^{n-1}T_{\alpha}^iX,S^1\right)&\ge& cd^H\left(\bigcup_{i=0}^{q_{n_k}+n}\{T_{\alpha}^i(0)\}\cup\bigcup_{i=0}^{n-1}\{T_{\alpha}^{i+q_{n_{k+1}}}(0)\},S^1\right)\nonumber\\
&\ge& \frac{1}{2}cd^H\left(\bigcup_{i=0}^{q_{n_k}+n}\{T_{\alpha}^i(0)\},S^1\right)\nonumber\\&\ge& \frac{1}{2}cd^H\left(\bigcup_{i=0}^{3n}\{T_{\alpha}^i(0)\},S^1\right).\nonumber
\end{eqnarray}
Thus $$n\cdot d^H\left(\bigcup_{i=0}^{n-1}T_{\alpha}^iX,S^1\right)\ge \frac{1}{2}cn\cdot d^H\left(\bigcup_{i=0}^{3n}\{T_{\alpha}^i(0)\},S^1\right).$$
Hence by (\ref{e3}), $$\liminf_nn\cdot d^H\left(\bigcup_{i=0}^{n-1}T_{\alpha}^iX,S^1\right)\ge \frac{c}{6}\liminf_n3n\cdot d^H\left(\bigcup_{i=0}^{3n}\{T_{\alpha}^i(0)\},S^1\right)>0.$$
\end{proof}

\section{Concluding remarks}

We are not aware of a direct proof of Theorem \ref{t2} without using the property of circle rotation. It will be good to find one, as it possibly can be generalized to prove that for all IETs. In fact, another motivation of this work is the following two conjectures.
\begin{conjecture}
If $X$ is an infinite subset of $I=[0,1]$, then for a.e. IET $T$, $$\inf_n d^H(T^nX,I)=0.$$
\end{conjecture}

\begin{conjecture}
If $X$ is an infinite subset of $I$, then for a.e. IET $T$, $$\inf_nn\cdot d^H\left(\bigcup_{k=0}^{n-1}T^kX,I\right)=0.$$
\end{conjecture}

It is also natural to ask
\begin{question}
If $X$ is an infinite subset of $I$, is it true that for a.e. IET $T$, $$\sup_nn\cdot d^H\left(\bigcup_{k=0}^{n-1}T^kX,I\right)=0?$$
\end{question}

As we illustrated before, it will be interesting to know whether one can find a $\mathbb Z$ action with Glasner property or Q.D. property. It will be surprising to really have such an example with either property. We hope to make further progress towards these questions in the subsequent work.

\section*{Acknowledgements} The author is grateful to Federico Rodriguez Hertz for drawing his attention to cw-expansivity and some discussions on the existence of Glasner property for $\mathbb Z$ action, and to Zhiren Wang for carefully reading a preliminary draft.

%
%
%
%
%
%
%

\text{\quad}\\

\textsc{Department of Mathematics, The Pennsylvania State University, University Park, PA 16802, USA}

{Email: cud159@psu.edu}

\end{document}